\documentclass[psamsfonts,english]{amsart}

    \usepackage[T1]{fontenc}
    \usepackage{babel,amsmath,amssymb,amsthm,url,mathtools}
    \usepackage[mathcal]{eucal}

    \newtheorem{theorem}{Theorem}
    \newtheorem{corollary}[theorem]{Corollary}
    
    \theoremstyle{definition}
    \newtheorem{definition}[theorem]{Definition}
    \theoremstyle{remark}
    \newtheorem{remark}[theorem]{Remark}
    \newtheorem{example}[theorem]{Example}

    \newcommand{\FF}{\mathbb{F}}
    \newcommand{\NN}{\mathbb{N}}
    \newcommand{\QQ}{\mathbb{Q}}
    \newcommand{\ZZ}{\mathbb{Z}}
    
    \newcommand{\CC}{\mathbb{C}}
    \newcommand{\PP}{\mathbb{P}}
    \newcommand{\fullED}{\varkappa}
    \newcommand{\jac}{\mathcal{J}_{C}}
    \newcommand{\heltal}[1]{\mathfrak{O}_{#1}}
    \newcommand{\frob}{\varphi}

    \DeclareMathOperator{\Div}{Div}
    \DeclareMathOperator{\divisor}{div}
    \DeclareMathOperator{\Mat}{Mat}
    \DeclareMathOperator{\diag}{diag}

    \newcommand{\nummerkontrol}[1]{}

\begin{document}

\title{Non-Cyclic Subgroups of Jacobians of Genus~Two~Curves}

\author[C.R. Ravnshøj]{Christian Robenhagen Ravnshøj}

\address{Department of Mathematical Sciences \\
University of Aarhus \\
Ny Munkegade \\
Building 1530 \\
DK-8000 Aarhus C}

\email{cr@imf.au.dk}

\thanks{Research supported in part by a PhD grant from CRYPTOMAThIC}



% 11G20 Arithmetic algebraic geometry: Curves over finite and local fields
% 11T71 Finite fields and commutative rings (numbertheoretic aspects): Algebraic coding theory; cryptography
% 14G50 Arithmetic problems. Diophantine geometry: Applications to coding theory and cryptography
% 14H45 Curves: Special curves and curves of low genus

\begin{abstract}
Let~$E$ be an elliptic curve defined over a finite field. Balasubramanian and Koblitz have proved that if the
$\ell^\text{th}$ roots of unity $\mu_\ell$ is not contained in the ground field, then a field extension of the ground
field contains $\mu_\ell$ if and only if the $\ell$-torsion points of~$E$ are rational over the same field extension.
We generalize this result to Jacobians of genus two curves. In particular, we show that the Weil- and the Tate-pairing
are non-degenerate over the \emph{same} field extension of the ground field.

From this generalization we get a complete description of the $\ell$-torsion subgroups of Jacobians of supersingular
genus two curves. In particular, we show that for $\ell>3$, the $\ell$-torsion points are rational over a field
extension of degree at most $24$.
\end{abstract}

\maketitle

\section{Introduction}

In \cite{koblitz87}, Koblitz described how to use elliptic curves to construct a public key cryptosystem. To get a more
general class of curves, and possibly larger group orders, Koblitz \cite{koblitz89} then proposed using Jacobians of
hyperelliptic curves. After Boneh and Franklin \cite{boneh-franklin} proposed an identity based cryptosystem by using
the Weil-pairing on an elliptic curve, pairings have been of great interest to cryptography~\cite{galbraith05}. The
next natural step was to consider pairings on Jacobians of hyperelliptic curves. Galbraith \emph{et
al}~\cite{galbraith07} survey the recent research on pairings on Jacobians of hyperelliptic curves.

The pairing in question is usually the Weil- or the Tate-pairing; both pairings can be computed with Miller's algorithm
\cite{miller-algorithm}. The Tate-pairing can be computed more efficiently than the Weil-pairing, cf.
\cite{galbraith01}. Let~$C$ be a smooth curve defined over a finite field $\FF_q$, and let~$\jac$ be the Jacobian
of~$C$. Let $\ell$ be a prime number dividing the number of $\FF_q$-rational points on the Jacobian, and let $k$ be the
multiplicative order of $q$ modulo $\ell$. By \cite{hess}, the Tate-pairing is non-degenerate
on~$\jac(\FF_{q^k})[\ell]$. By \cite[Proposition~8.1, p.~96]{sil}, the Weil-pairing is non-degenerate on~$\jac[\ell]$.
So if~$\jac[\ell]$ is not contained in~$\jac(\FF_{q^k})$, then the Tate pairing is non-degenerate over a possible
smaller field extension than the Weil-pairing. For elliptic curves, in most cases relevant to cryptography, the
Weil-pairing and the Tate-pairing are non-degenerate over the same field: let~$E$ be an elliptic curve defined
over~$\FF_p$, and consider a prime number $\ell$ dividing the number of $\FF_p$-rational points on~$E$. Balasubramanian
and Koblitz \cite{balasubramanian} proved that
    \begin{equation}\label{eq:embeddingEllipticCurves}
    \text{\emph{if $\ell\nmid p-1$, then $E[\ell]\subseteq E(\FF_{p^k})$ if and only if $\ell\mid p^k-1$.}}
    \end{equation}
By Rubin and Silverberg \cite{rubin-silverberg07}, this result also holds for Jacobians of genus two curves in the
following sense: \emph{if $\ell\nmid p-1$, then the Weil-pairing is non-degenerate on $U\times V$, where
$U=\jac(\FF_p)[\ell]$, $V=\ker(\frob-p)\cap\jac[\ell]$ and $\frob$ is the $p$-power Frobenius endomorphism on~$\jac$}.

The result~\eqref{eq:embeddingEllipticCurves} can also be stated as: \emph{if $\ell\nmid p-1$, then
$E(\FF_{p^k})[\ell]$ is bicyclic if and only if $\ell\mid p^k-1$}. In \cite{ravnshoj}, the author generalized this
result to certain CM reductions of Jacobians of genus two curves. In this paper, we show that in most cases, this
result in fact holds for Jacobians of \emph{any} genus two curves. More precisely, the following theorem is
established.

%samme nummerering af theorem som i afsnittet, hvor den bevises. Pas på! Dette er en MANUEL løsning...
\setcounter{section}{5}\setcounter{theorem}{5}

\begin{theorem}\nummerkontrol{\ref{theorem:main}}
Consider a genus two curve~$C$ defined over a finite field~$\FF_q$. Write the characteristic polynomial of the
$q^m$-power Frobenius endomorphism of the Jacobian~$\jac$ as
    $$P_m(X)=X^4+2\sigma X^3+(2q^m+\sigma^2-\tau)X^2+2\sigma q^mX+q^{2m},$$
where $2\sigma,4\tau\in\ZZ$. Let $\ell$ be an odd prime number dividing the number of $\FF_q$-rational points on
$\jac$, and with $\ell\nmid q$ and $\ell\nmid q-1$. If $\ell\nmid 4\tau$, then
    \begin{enumerate}
    \item~$\jac(\FF_{q^m})[\ell]$ is of rank at most two as a $\ZZ/\ell\ZZ$-module, and
    \item~$\jac(\FF_{q^m})[\ell]$ is bicyclic if and only if $\ell$ divides $q^m-1$.
    \end{enumerate}
\end{theorem}

If $\ell$ is a large prime number, then most likely $\ell\nmid 4\tau$, and Theorem~\ref{theorem:main} applies. In the
special case $\ell\mid 4\tau$ we get the following result.

\setcounter{section}{5}\setcounter{theorem}{6}

\begin{theorem}\nummerkontrol{\ref{theorem:main:supplement}}
Let notation be as in Theorem~\ref{theorem:main}. Furthermore, let $\omega_m$ be a $q^m$-Weil number of~$\jac$
(cf.~definition~\ref{definition:WeilNumber}), and assume that $\ell$ is unramified in $K=\QQ(\omega_m)$. Now assume
that~\mbox{$\ell\mid 4\tau$}. Then the following holds.
 \begin{enumerate}
 \item If $\omega_m\in\ZZ$, then $\ell\mid q^m-1$ and~$\jac[\ell]\subseteq\jac(\FF_{q^m})$.
 \item If $\omega_m\notin\ZZ$, then $\ell\nmid q^m-1$, $\jac(\FF_{q^m})[\ell]\simeq(\ZZ/\ell\ZZ)^2$ and
      ~$\jac[\ell]\subseteq\jac(\FF_{q^{mk}})$ if and only if $\ell\mid q^{mk}-1$.
 \end{enumerate}
\end{theorem}

By Theorem~\ref{theorem:main} and \ref{theorem:main:supplement} we get the following corollary.

\setcounter{section}{5}\setcounter{theorem}{9}

\begin{corollary}\nummerkontrol{\ref{corollary:WeilNonDegenerate}}
Consider a genus two curve~$C$ defined over a finite field~$\FF_q$. Let $\ell$ be an odd prime number dividing the
number of $\FF_q$-rational points on the Jacobian~$\jac$, and with $\ell\nmid q$. Let $q$ be of multiplicative order
$k$ modulo $\ell$. If $\ell\nmid q-1$, then the Weil-pairing is non-degenerate on
$\jac(\FF_{q^k})[\ell]\times\jac(\FF_{q^k})[\ell]$.
\end{corollary}

For the $2$-torsion part, we prove the following theorem.

\setcounter{section}{5}\setcounter{theorem}{10}

\begin{theorem}\nummerkontrol{\ref{theorem:even}}
Consider a genus two curve~$C$ defined over a finite field~$\FF_q$ of odd characteristic. Let
    $$P_m(X)=X^4+sX^3+tX^2+sq^mX+q^{2m}$$
be the characteristic polynomial of the $q^m$-power Frobenius endomorphism of the Jacobian~$\jac$. Assume
$|\jac(\FF_{q^m})|$ is even. Then
    $$
    \jac[2]\subseteq
            \begin{cases}
            \jac(\FF_{q^{4m}}), & \text{if $s$ is even;} \\
            \jac(\FF_{q^{6m}}), & \text{if $s$ is odd.}
            \end{cases}
    $$
\end{theorem}

Now consider a supersingular genus two curve~$C$ defined over $\FF_q$; cf.~section~\ref{sec:supersingularCurves}.
Again, let $\ell$ be a prime number dividing the number of $\FF_q$-rational points on the Jacobian and let~$k$ be the
multiplicative order of $q$ modulo $\ell$. We know that $k\leq 12$, cf. Galbraith~\cite{galbraith01} and Rubin and
Silverberg~\cite{rubin-silverberg02}. If $\ell^2\nmid |\jac(\FF_q)|$, then in many cases
$\jac[\ell]\subseteq\jac(\FF_{q^k})$, cf. Stichtenoth~\cite{stichtenoth}. Zhu~\cite{zhu} gives a complete description
of the subgroup of $\FF_q$-rational points on the Jacobian. Using Theorem~\ref{theorem:main} we get the following
explicit description of the $\ell$-torsion subgroup of the Jacobian of a supersingular genus two curve.

%samme nummerering af theorem som i afsnittet, hvor den bevises. Pas på! Dette er en MANUEL løsning...
\setcounter{section}{6}\setcounter{theorem}{13}

\begin{theorem}\nummerkontrol{\ref{theorem:TEDss}}
Consider a supersingular genus two curve~$C$ defined over~$\FF_q$. Let~$\ell$ be a prime number dividing the number of
$\FF_q$-rational points on the Jacobian~$\jac$, and with $\ell\nmid q$. Depending on the cases in
table~\ref{table:conditions} we get the following properties of~$\jac$.
    \begin{description}
    \item[Case~\textsc{i}] $-q^2\equiv q^4\equiv 1\pmod{\ell}$ and~$\jac[\ell]\subseteq\jac(\FF_{q^4})$. If $\ell\neq 2$, then~$\jac(\FF_q)[\ell]$ is cyclic.
    \item[Case~\textsc{ii}] $q^3\equiv 1\pmod{\ell}$, $\jac[\ell]\subseteq\jac(\FF_{q^6})$ and~$\jac(\FF_q)$ is cyclic.
    If $\ell\neq 3$, then $q\not\equiv 1\pmod{\ell}$.
    \item[Case~\textsc{iii}] $-q^3\equiv q^6\equiv 1\pmod{\ell}$ and~$\jac[\ell]\subseteq\jac(\FF_{q^6})$. If $\ell\neq 3$, then~$\jac(\FF_q)[\ell]$ is cyclic.
    \item[Case~\textsc{iv}] $q\not\equiv q^5\equiv 1\pmod{\ell}$, $\jac[\ell]\subseteq\jac(\FF_{q^{10}})$ and~$\jac(\FF_q)$ is cyclic.
    \item[Case~\textsc{v}] $q\not\equiv q^5\equiv 1\pmod{\ell}$, $\jac[\ell]\subseteq\jac(\FF_{q^{10}})$ and~$\jac(\FF_q)$ is cyclic.
    \item[Case~\textsc{vi}] $-q^6\equiv q^{12}\equiv 1\pmod{\ell}$, $\jac[\ell]\subseteq\jac(\FF_{q^{24}})$ and~$\jac(\FF_q)$ is cyclic.
    \item[Case~\textsc{vii}] $q\equiv 1\pmod{\ell}$ and~$\jac[\ell]\subseteq\jac(\FF_{q^2})$. If $\ell\neq 2$, then~$\jac(\FF_q)[\ell]$ is bicyclic.
    \item[Case~\textsc{viii}] $-q\equiv q^2\equiv 1\pmod{\ell}$ and~$\jac[\ell]\subseteq\jac(\FF_{q^2})$. If $\ell\neq 2$, then~$\jac(\FF_q)[\ell]$ is bicyclic.
    \item[Case~\textsc{ix}] If $\ell\neq 3$, then $q\not\equiv q^3\equiv 1\pmod{\ell}$, $\jac[\ell]\subseteq\jac(\FF_{q^3})$ and~$\jac(\FF_q)[\ell]$ is bicyclic.
    \end{description}
\end{theorem}

In particular, it follows from Theorem~\ref{theorem:TEDss} that if $\ell>3$, then the $\ell$-torsion points on the
Jacobian~$\jac$ of a supersingular genus two curve defined over~$\FF_q$ are rational over a field extension of~$\FF_q$
of degree at most~$24$, and~$\jac(\FF_q)[\ell]$ is of rank at most two as a $\ZZ/\ell\ZZ$-module.

\setcounter{section}{1} \setcounter{theorem}{0}

\subsection*{Assumption}

In this paper, a \emph{curve} is an irreducible nonsingular projective variety of dimension one.

\section{Genus two curves}\label{sec:HyperellipticCurves}

A hyperelliptic curve is a projective curve $C\subseteq\PP^n$ of genus at least two with a separable, degree two
morphism $\phi:C\to\PP^1$. It is well known, that any genus two curve is hyperelliptic. Throughout this paper, let~$C$
be a curve of genus two defined over a finite field~$\FF_q$ of characteristic~$p$. By the Riemann-Roch Theorem there
exists a birational map \mbox{$\psi:C\to\PP^2$}, mapping~$C$ to a curve given by an equation of the form
    $$y^2+g(x)y=h(x),$$
where $g,h\in\FF_q[x]$ are of degree $\deg(g)\leq 3$ and $\deg(h)\leq 6$; cf.~\cite[chapter~1]{cassels}.

The set of principal divisors $\mathcal{P}(C)$ on~$C$ constitutes a subgroup of the degree zero divisors $\Div_0(C)$.
The Jacobian~$\jac$ of~$C$ is defined as the quotient
    $$\jac=\Div_0(C)/\mathcal{P}(C).$$
{\samepage Let $\ell\neq p$ be a prime number. The $\ell^n$-torsion subgroup~$\jac[\ell^n]\subseteq\jac$ of points of
order dividing $\ell^n$ is a $\ZZ/\ell^n\ZZ$-module of rank four, i.e.
    \begin{equation*}\label{eq:J-struktur}
    \jac[\ell^n]\simeq\ZZ/\ell^n\ZZ\times\ZZ/\ell^n\ZZ\times\ZZ/\ell^n\ZZ\times\ZZ/\ell^n\ZZ;
    \end{equation*}
cf.~\cite[Theorem~6, p.~109]{lang59}.}

The multiplicative order $k$ of $q$ modulo $\ell$ plays an important role in cryptography, since the (reduced)
Tate-pairing is non-degenerate over $\FF_{q^k}$; cf.~\cite{hess}.

\begin{definition}[Embedding degree]
Consider a prime number $\ell\neq p$ dividing the number of $\FF_q$-rational points on the Jacobian~$\jac$. The
embedding degree of~$\jac(\FF_q)$ with respect to $\ell$ is the least number $k$, such that $q^k\equiv 1\pmod{\ell}$.
\end{definition}

Closely related to the embedding degree, we have the \emph{full} embedding degree.

\begin{definition}[Full embedding degree]
Consider a prime number $\ell\neq p$ dividing the number of $\FF_q$-rational points on the Jacobian~$\jac$. The full
embedding degree of~$\jac(\FF_q)$ with respect to $\ell$ is the least number $\fullED$, such that
$\jac[\ell]\subseteq\jac(\FF_{q^\fullED})$.
\end{definition}

\begin{remark}
If~$\jac[\ell]\subseteq\jac(\FF_{q^\fullED})$, then $\ell\mid q^\fullED-1$; cf.~\cite[Corollary~5.77, p.~111]{hhec}.
Hence, the full embedding degree is a multiple of the embedding degree.
\end{remark}

A priori, the Weil-pairing is only non-degenerate over $\FF_{q^\fullED}$. But in fact, as we shall see, the
Weil-pairing is also non-degenerate over $\FF_{q^k}$.

\section{The Weil- and the Tate-pairing}

Let $\FF$ be an algebraic extension of~$\FF_q$. Let $x\in\jac(\FF)[\ell]$ and $y=\sum_ia_i P_i\in\jac(\FF)$ be divisors
with disjoint supports, and let $\bar{y}\in\jac(\FF)/\ell\jac(\FF)$ denote the divisor class containing the
divisor~$y$. Furthermore, let $f_x\in\FF(C)$ be a rational function on~$C$ with divisor $\divisor(f_x)=\ell x$. Set
$f_x(y)=\prod_if(P_i)^{a_i}$. Then $e_\ell(x,\bar{y})=f_x(y)$ is a well-defined pairing
    $$e_\ell:\jac(\FF)[\ell]\times\jac(\FF)/\ell\jac(\FF)\longrightarrow\FF^\times/(\FF^\times)^\ell,$$
it is called the \emph{Tate-pairing}; cf.~\cite{galbraith05}. Raising the result to the
power~$\frac{|\FF^\times|}{\ell}$ gives a well-defined element in the subgroup $\mu_\ell\subseteq\bar\FF$ of the
$\ell^{\mathrm{th}}$ roots of unity. This pairing
    $$\hat{e}_\ell:\jac(\FF)[\ell]\times\jac(\FF)/\ell\jac(\FF)\longrightarrow\mu_\ell$$
is called the \emph{reduced} Tate-pairing. If the field~$\FF$ is finite and contains the $\ell^\mathrm{th}$ roots of
unity, then the Tate-pairing is bilinear and non-degenerate; cf.~\cite{hess}.

Now let $x,y\in\jac[\ell]$ be divisors with disjoint support. The Weil-pairing
    $$e_\ell:\jac[\ell]\times\jac[\ell]\to\mu_\ell$$
is then defined by $e_\ell(x,y)=\frac{\hat{e}_\ell(x,\bar{y})}{\hat{e}_\ell(y,\bar{x})}$. The Weil-pairing is bilinear,
anti-symmetric and non-degenerate on~$\jac[\ell]\times\jac[\ell]$; cf.~\cite{miller}.

\section{Matrix representation of the endomorphism ring}

An endomorphism $\psi:\jac\to\jac$ induces a linear map $\bar{\psi}:\jac[\ell]\to\jac[\ell]$ by restriction. Hence,
$\psi$ is represented by a matrix $M\in\Mat_4(\ZZ/\ell\ZZ)$ on~$\jac[\ell]$. Let $f\in\ZZ[X]$ be the characteristic
polynomial of $\psi$ (see~\cite[pp.~109--110]{lang59}), and let $\bar{f}\in(\ZZ/\ell\ZZ)[X]$ be the characteristic
polynomial of $\bar{\psi}$. Then $f$ is a monic polynomial of degree four, and by \cite[Theorem~3, p.~186]{lang59},
    \begin{equation*}\label{eq:KarPolKongruens}
    f(X)\equiv\bar{f}(X)\pmod{\ell}.
    \end{equation*}

Since~$C$ is defined over~$\FF_q$, the mapping $(x,y)\mapsto (x^q,y^q)$ is a morphism on~$C$. This morphism induces the
$q$-power Frobenius endo\-morphism $\frob$ on the Jacobian~$\jac$. Let $P(X)$ be the characteristic polynomial
of~$\frob$. $P(X)$ is called the \emph{Weil polynomial} of~$\jac$, and
    $$|\jac(\FF_q)|=P(1)$$
by the definition of $P(X)$ (see~\cite[pp.~109--110]{lang59}); i.e. the number of~$\FF_q$-rational points on the
Jacobian is $P(1)$.

\begin{definition}[Weil number]\label{definition:WeilNumber}
Let notation be as above. Let $P_m(X)$ be the characteristic polynomial of the $q^m$-power Frobenius
endomorphism~$\frob_m$ on~$\jac$. Consider a number $\omega_m\in\CC$ with $P_m(\omega_m)=0$. If $P_m(X)$ is reducible,
assume furthermore that $\omega_m$ and $\frob_m$ are roots of the same irreducible factor of $P_m(X)$. We
identify~$\frob_m$ with~$\omega_m$, and we call~$\omega_m$ a \emph{$q^m$-Weil~number} of~$\jac$.
\end{definition}

\begin{remark}\label{rem:P_reducible}
A $q^m$-Weil~number is not necessarily uniquely determined. In general, $P_m(X)$ is irreducible, in which case~$\jac$
has four $q^m$-Weil~numbers.

Assume $P_m(X)$ is reducible. Write $P_m(X)=f(X)g(X)$, where $f,g\in\ZZ[X]$ are of degree at least one. Since
$P_m(\frob_m)=0$, either $f(\frob_m)=0$ or $g(\frob_m)=0$; if not, then either $f(\frob_m)$ or $g(\frob_m)$ has
infinite kernel, i.e. is not an endomorphism of~$\jac$. So a $q^m$-Weil number is well-defined.
\end{remark}

\section{Non-cyclic torsion}\label{sec:properties}

Consider a genus two curve~$C$ defined over a finite field~$\FF_q$. Let $P_m(X)$ be the characteristic polynomial of
the $q^m$-power Frobenius endomorphism~$\frob_m$ of the Jacobian~$\jac$. $P_m(X)$ is of the form
$P_m(X)=X^4+sX^3+tX^2+sq^mX+q^{2m}$, where $s,t\in\ZZ$. Let $\sigma=\frac{s}{2}$ and $\tau=2q^m+\sigma^2-t$. Then
    $$P_m(X)=X^4+2\sigma X^3+(2q^m+\sigma^2-\tau)X^2+2\sigma q^mX+q^{2m},$$
and $2\sigma,4\tau\in\ZZ$.

\begin{theorem}\label{theorem:main}
Consider a genus two curve~$C$ defined over a finite field~$\FF_q$. Write the characteristic polynomial of the
$q^m$-power Frobenius endomorphism of the Jacobian~$\jac$ as
    $$P_m(X)=X^4+2\sigma X^3+(2q^m+\sigma^2-\tau)X^2+2\sigma q^mX+q^{2m},$$
where $2\sigma,4\tau\in\ZZ$. Let $\ell$ be an odd prime number dividing the number of $\FF_q$-rational points on
$\jac$, and with $\ell\nmid q$ and $\ell\nmid q-1$. If $\ell\nmid 4\tau$, then
    \begin{enumerate}
    \item~$\jac(\FF_{q^m})[\ell]$ is of rank at most two as a $\ZZ/\ell\ZZ$-module, and
    \item~$\jac(\FF_{q^m})[\ell]$ is bicyclic if and only if $\ell$ divides $q^m-1$.
    \end{enumerate}
\end{theorem}

\begin{proof}
Let $\bar{P}_m\in(\ZZ/\ell\ZZ)[X]$ be the characteristic polynomial of the restriction of $\frob_m$ to~$\jac[\ell]$.
Since $\ell$ divides $|\jac(\FF_q)|$, $1$ is a root of $\bar{P}_m$. Assume that $1$ is a root of $\bar{P}_m$ of
multiplicity~$\nu$. Since the roots of $\bar{P}_m$ occur in pairs $(\alpha,q^m/\alpha)$, also $q^m$ is a root of
$\bar{P}_m$ of multiplicity~$\nu$.

If~$\jac(\FF_{q^m})[\ell]$ is of rank three as a $\ZZ/\ell\ZZ$-module, then $\ell$ divides $ q^m-1$ by
\cite[Proposition~5.78, p.~111]{hhec}. Choose a basis $\mathcal{B}$ of $\jac[\ell]$. With respect to $\mathcal{B}$,
$\frob_m$ is represented by a matrix of the form
    $$M=\begin{bmatrix}
        1 & 0 & 0 & m_1 \\
        0 & 1 & 0 & m_2 \\
        0 & 0 & 1 & m_3 \\
        0 & 0 & 0 & m_4
        \end{bmatrix}.
    $$
Now, $m_4=\det M\equiv\deg\frob_m=q^{2m}\equiv 1\pmod{\ell}$. Hence, $\bar{P}_m(X)=(X-1)^4$. By comparison of
coefficients it follows that $4\tau\equiv 0\pmod{\ell}$, and we have a contradiction. So~$\jac(\FF_{q^m})[\ell]$ is of
rank at most two as a $\ZZ/\ell\ZZ$-module.

Now assume that~$\jac(\FF_{q^m})[\ell]$ is bicyclic. If $q^m\not\equiv 1\pmod{\ell}$, then $1$ is a root of $\bar{P}_m$
of multiplicity two, i.e. $\bar{P}_m(X)=(X-1)^2(X-q^m)^2$. But then it follows by comparison of coefficients that
$4\tau\equiv 0\pmod{\ell}$, and we have a contradiction. So $q^m\equiv 1\pmod{\ell}$, i.e. $\ell$ divides $q^m-1$. On
the other hand, if $\ell$ divides $q^m-1$, then the Tate-pairing is non-degenerate on~$\jac(\FF_{q^m})[\ell]$, i.e.
$\jac(\FF_{q^m})[\ell]$ must be of rank at least two as a $\ZZ/\ell\ZZ$-module.  So~$\jac(\FF_{q^m})[\ell]$ is
bicyclic.
\end{proof}

If $\ell$ is a large prime number, then most likely $\ell\nmid 4\tau$, and Theorem~\ref{theorem:main} applies. In the
special case $\ell\mid 4\tau$ we get the following result.

\begin{theorem}\label{theorem:main:supplement}
Let notation be as in Theorem~\ref{theorem:main}. Furthermore, let $\omega_m$ be a $q^m$-Weil number of~$\jac$, and
assume that $\ell$ is unramified in $K=\QQ(\omega_m)$. Now assume that~\mbox{$\ell\mid 4\tau$}. Then the following
holds.
 \begin{enumerate}
 \item If $\omega_m\in\ZZ$, then $\ell\mid q^m-1$ and~$\jac[\ell]\subseteq\jac(\FF_{q^m})$.
 \item If $\omega_m\notin\ZZ$, then $\ell\nmid q^m-1$, $\jac(\FF_{q^m})[\ell]\simeq(\ZZ/\ell\ZZ)^2$ and
      ~$\jac[\ell]\subseteq\jac(\FF_{q^{mk}})$ if and only if $\ell\mid q^{mk}-1$.
 \end{enumerate}
\end{theorem}

\begin{remark}
A prime number $\ell$ is unramified in~$K$ if and only if $\ell$ divides the discriminant of the field extension
$K/\QQ$; see e.g. \cite[Theorem~2.6, p.~199]{neukirch}. Hence, almost any prime number $\ell$ is unramified in~$K$. In
particular, if~$\ell$ is large, then $\ell$ is unramified in~$K$.
\end{remark}

The special case of Theorem~\ref{theorem:main:supplement} \emph{does} occur; cf. the following
example~\ref{example:SpecialCase}.

\begin{example}\label{example:SpecialCase}
Consider the polynomial $P(X)=(X^2+X+3)^2\in\QQ[X]$. By \cite{maisner-nart} and \cite{howeNartRitzenthaler} it follows
that $P(X)$ is the Weil polynomial of the Jacobian of a genus two curve~$C$ defined over $\FF_3$. The number of
$\FF_3$-rational points on the Jacobian is $P(1)=25$, so $\ell=5$ is an odd prime divisor of $|\jac(\FF_3)|$ not
dividing $q=p=3$. Notice that $P(X)\equiv X^4+2\sigma X^3+(2p+\sigma^2)X^2+2\sigma pX+p\pmod{\ell}$ with $\sigma=1$.
The complex roots of $P(X)$ are given by $\omega=\frac{-1+\sqrt{-11}}{2}$ and $\bar{\omega}$, and $\ell$ is unramified
in $K=\QQ(\omega)$. Since $3$ is a generator of $(\ZZ/5\ZZ)^\times$, it follows by
Theorem~\ref{theorem:main:supplement} that~$\jac(\FF_3)\simeq(\ZZ/\ell\ZZ)^2$ and~$\jac[\ell]\subseteq\jac(\FF_{81})$.
\end{example}

By Theorem~\ref{theorem:main} and \ref{theorem:main:supplement} we get the following corollary.

\begin{corollary}\label{corollary:WeilNonDegenerate}
Consider a genus two curve~$C$ defined over a finite field~$\FF_q$. Let $\ell$ be an odd prime number dividing the
number of $\FF_q$-rational points on the Jacobian~$\jac$, and with $\ell\nmid q$. Let $q$ be of multiplicative order
$k$ modulo $\ell$. If $\ell\nmid q-1$, then the Weil-pairing is non-degenerate on
$\jac(\FF_{q^k})[\ell]\times\jac(\FF_{q^k})[\ell]$.
\end{corollary}

\begin{proof}
Let
    $$P_k(X)=X^4+2\sigma X^3+(2q^k+\sigma^2-\tau)X^2+2\sigma q^kX+q^{2k}$$
be the characteristic polynomial of the $q^k$-power endomorphism on the Jacobian~$\jac$. If $\ell\mid 4\tau$, then
$\jac[\ell]=\jac(\FF_{q^k})[\ell]$ by Theorem~\ref{theorem:main:supplement}, and the corollary follows.

Assume $\ell\nmid 4\tau$. Let $U=\jac(\FF_q)[\ell]$ and $V=\ker(\frob-q)\cap\jac[\ell]$, where $\frob$ is the $q$-power
Frobenius endomorphism on~$\jac$. Then the Weil-pairing $e_W$ is non-degenerate on $U\times V$ by
\cite{rubin-silverberg07}. By Theorem~\ref{theorem:main}, we know that
$V=\jac(\FF_{q^k})[\ell]\setminus\jac(\FF_q)[\ell]$ and that
    $$\jac(\FF_{q^k})[\ell]\simeq U\oplus V\simeq\ZZ/\ell\ZZ\times\ZZ/\ell\ZZ.$$
Now let $x\in\jac(\FF_{q^k})[\ell]$ be an arbitrary $\FF_{q^k}$-rational point of order $\ell$. Write $x=x_U+x_V$,
where $x_U\in U$ and $x_V\in V$. Choose $y\in V$ and $z\in U$, such that $e_W(x_U,y)\neq 1$ and $e_W(x_V,z)\neq 1$. We
may assume that $e_W(x_U,y)e_W(x_V,z)\neq 1$; if not, replace $z$ with $2z$. Since the Weil-pairing is anti-symmetric,
$e_W(x_U,z)=e_W(x_V,y)=1$. Hence,
    $$e_W(x,y+z)=e_W(x_U,y)e_W(x_V,z)\neq 1.$$
\end{proof}

\begin{proof}[Proof of Theorem~\ref{theorem:main:supplement}]
We see that
    $$P_m(X)\equiv (X^2+\sigma X+q^m)^2 \pmod{\ell};$$
since $P_m(1)\equiv 0\pmod{\ell}$, it follows that
    $$
    P_m(X)\equiv (X-1)^2(X-q^m)^2\pmod{\ell}.
    $$

Assume at first that $P_m(X)$ is irreducible in $\QQ[X]$. Let~$\heltal{K}$ denote the ring of integers of~$K$.
By~\cite[Proposition~8.3, p.~47]{neukirch}, it follows that $\ell\heltal{K}=\mathfrak{L}_1^2\mathfrak{L}_2^2$, where
$\mathfrak{L}_1=(\ell,\omega_m-1)\heltal{K}$ and $\mathfrak{L}_2=(\ell,\omega_m-q)\heltal{K}$. In particular, $\ell$
ramifies in $K$, and we have a contradiction. So $P_m(X)$ is reducible in $\QQ[X]$.

Let $f\in\ZZ[X]$ be the minimal polynomial of $\omega_m$. If $\deg f=3$, then it follows as above that $\ell$ ramifies
in~$K$. So $\deg f<3$. Assume that $\deg f=1$, i.e. that $\omega_m\in\ZZ$. Since $\omega_m^2=q^m$, we know that
$\omega_m=\pm q^{m/2}$. So $f(X)=X\mp q^{m/2}$. Since $f(X)$ divides $P(X)$ in $\ZZ[X]$, either $f(X)\equiv
X-1\pmod{\ell}$ or $f(X)\equiv X-q^m\pmod{\ell}$. It follows that $q^m\equiv 1\pmod{\ell}$. Thus, $\omega_m\equiv\pm
1\pmod{\ell}$. If $\omega_m\equiv -1\pmod{\ell}$, then $\frob_m$ does not fix~$\jac(\FF_{q^m})[\ell]$. This is a
contradiction. Hence, $\omega_m\equiv 1\pmod{\ell}$. But then $\frob_m$ is the identity on~$\jac[\ell]$. Thus, if
$\omega_m\in\ZZ$, then~$\jac[\ell]\subseteq\jac(\FF_{q^m})$.

Assume $\omega_m\notin\ZZ$. Then $\deg f=2$. Since $f(X)$ divides $P(X)$ in $\ZZ[X]$, it follows that
    $$f(X)\equiv (X-1)(X-q^m)\pmod{\ell};$$
to see this, we merely notice that if $f(X)$ is equivalent to the square of a polynomial modulo $\ell$, then $\ell$
ramifies in $K$. Notice also that if $q^m\equiv 1\pmod{\ell}$, then $\ell$ ramifies in $K$. So $q^m\not\equiv
1\pmod{\ell}$.

Now let $U=\ker(\frob_m-1)^2\cap\jac[\ell]$ and $V=\ker(\frob_m-q^m)^2\cap\jac[\ell]$. Then $U$ and $V$ are
$\frob_m$-invariant submodules of the $\ZZ/\ell\ZZ$-module~$\jac[\ell]$ of rank two, and~$\jac[\ell]\simeq U\oplus V$.
Now choose $x_1\in U$, such that $\frob_m(x_1)=x_1$, and expand this to a basis $(x_1,x_2)$ of $U$. Similarly, choose a
basis $(x_3,x_4)$ of $V$ with $\frob_m(x_3)=qx_3$. With respect to the basis $(x_1,x_2,x_3,x_4)$, $\frob_m$ is
represented by a matrix of the form
    $$
    M=\begin{bmatrix}
      1 & \alpha & 0 & 0 \\
      0 & 1 & 0 & 0 \\
      0 & 0 & q^m & \beta \\
      0 & 0 & 0 & q^m
      \end{bmatrix}.
    $$
Let $q^m$ be of multiplicative order $k$ modulo $\ell$. Notice that
    $$
    M^k=\begin{bmatrix}
      1 & k\alpha & 0 & 0 \\
      0 & 1 & 0 & 0 \\
      0 & 0 & 1 & k q^{m(k-1)}\beta \\
      0 & 0 & 0 & 1
      \end{bmatrix}.
    $$
Hence, the restriction of $\frob_m^k$ to~$\jac[\ell]$ has the characteristic polynomial $(X-1)^4$. Let $P_{mk}(X)$ be
the characteristic polynomial of the $q^{mk}$-power Frobenius endomorphism $\frob_{mk}=\frob_m^k$ of the Jacobian
$\jac$. Then
    \begin{equation*}\label{eq:P_mk(X)moduloL}
    P_{mk}(X)\equiv (X-1)^4\pmod{\ell}.
    \end{equation*}
Since $\omega_m$ is a $q^m$-Weil number of~$\jac$, we know that $\omega_m^k$ is a $q^{mk}$-Weil number of~$\jac$.
Assume $\omega_m^k\notin\QQ$. Then $K=\QQ(\omega_m^k)$. Let $h\in\ZZ[X]$ be the minimal polynomial of~$\omega_m^k$.
Then it follows that $h(X)\equiv (X-1)^2\pmod{\ell}$, and $\ell$ ramifies in $K$. So $\omega_m^k\in\QQ$, i.e. $h$ is of
degree one. But then $h(X)\equiv X-1\pmod{\ell}$, i.e. $\omega_m^k\equiv 1\pmod{\ell}$. So $\frob_m^k$ is the identity
map on~$\jac[\ell]$. Hence, $M^k=I$, i.e. $\alpha\equiv\beta\equiv 0\pmod{\ell}$. Thus, $\frob_m$ is represented by a
diagonal matrix $\diag(1,1,q^m,q^m)$ with respect to $(x_1,x_2,x_3,x_4)$. The theorem follows.
\end{proof}

For the $2$-torsion part, we get the following theorem.

{\samepage
\begin{theorem}\label{theorem:even}
Consider a genus two curve~$C$ defined over a finite field~$\FF_q$ of odd characteristic. Let
    $P_m(X)=X^4+sX^3+tX^2+sq^mX+q^{2m}$
be the characteristic polynomial of the $q^m$-power Frobenius endomorphism of the Jacobian~$\jac$. Assume
$|\jac(\FF_{q^m})|$ is even. Then
    $$
    \jac[2]\subseteq
            \begin{cases}
            \jac(\FF_{q^{4m}}), & \text{if $s$ is even;} \\
            \jac(\FF_{q^{6m}}), & \text{if $s$ is odd.}
            \end{cases}
    $$
\end{theorem}
}

\begin{proof}
Since $q$ is odd,
    $$P_m(X)\equiv X^4+sX^3+tX^2+sX+1\pmod{2}.$$
Assume at first that $s$ is even. Since $P_m(1)$ is even, it follows that $t$ is even; but then
    $$P_m(X)\equiv (X-1)^4\equiv X^4-1\pmod{2}.$$
Hence, $\jac[2]\subseteq\jac(\FF_{q^{4m}})$ in this case.

Now assume that $s$ is odd. Again $t$ must be even; but then
    $$P_m(X)\equiv (X^2-1)(X^2+X+1)\pmod{2}.$$
Since $f(X)=X^2+X+1$ has the complex roots $\xi=-\frac{1}{2}(1\pm i\sqrt{3})$, and $\xi^3=1$, it follows that
$\jac[2]\subseteq\jac(\FF_{q^{6m}})$ in this case.
\end{proof}

\section{Supersingular curves}\label{sec:supersingularCurves}

Consider a genus two curve~$C$ defined over a finite field~$\FF_q$ of characteristic $p$.~$C$~is called
\emph{supersingular}, if~$\jac$ has no $p$-torsion. From \cite{maisner-nart} we have the following theorem.

\begin{theorem}\label{theorem:supersingular}
Consider a polynomial $f\in\ZZ[X]$ of the form
    $$f(X)=f_{s,t}(X)=X^4+sX^3+tX^2+sqX+q^2,$$
where $q=p^a$. If $f$ is the Weil polynomial of the Jacobian of a supersingular genus two curve defined over the finite
field~$\FF_q$, then $(s,t)$ belongs to table~\ref{table:conditions}.
\end{theorem}

\begin{table}[!ht]
\centering \caption{Conditions for $f=X^4+sX^3+tX^2+sqX+q^2$ to be the Weil polynomial of the Jacobian of a
supersingular genus two curve defined over~$\FF_q$, where~$q=p^a$.}\label{table:conditions}
\begin{tabular}{|l|l|l|}
    \hline
    Case & $(s,t)$              & Condition \\ \hline\hline
    \textsc{i}    & $(0,0)$              & $a$ odd, $p\neq 2$, or $a$ even, $p\not\equiv 1\pmod{8}$. \\ \hline
    \textsc{ii}   & $(0,q)$              & $a$ odd. \\ \hline
    \textsc{iii}  & $(0,-q)$             & $a$ odd, $p\neq 3$, or $a$ even, $p\not\equiv 1\pmod{12}$. \\ \hline
    \textsc{iv}   & $(\pm\sqrt{q},q)$    & $a$ even, $p\not\equiv 1\pmod{5}$. \\ \hline
    \textsc{v}    & $(\pm\sqrt{5q},3q)$  & $a$ odd, $p=5$. \\ \hline
    \textsc{vi}   & $(\pm\sqrt{2q},q)$,  & $a$ odd, $p=2$. \\ \hline
    \textsc{vii}  & $(0,-2q)$            & $a$ odd. \\ \hline
    \textsc{viii} & $(0,2q)$             & $a$ even, $p\equiv 1\pmod{4}$. \\ \hline
    \textsc{ix}   & $(\pm 2\sqrt{q},3q)$ & $a$ even, $p\equiv 1\pmod{3}$. \\ \hline
    \end{tabular}
\end{table}

\begin{remark}
By \cite{howeNartRitzenthaler}, in each of the cases in table~\ref{table:conditions} we can find a~$q$ such that
$f_{s,t}(X)$ is the Weil polynomial of the Jacobian of a supersingular genus two curve defined over~$\FF_q$.
\end{remark}

Using Theorem~\ref{theorem:main}, \ref{theorem:main:supplement} and \ref{theorem:supersingular} we get the following
explicit description of the $\ell$-torsion subgroup of the Jacobian of a supersingular genus two curve.

\begin{theorem}\label{theorem:TEDss}
Consider a supersingular genus two curve~$C$ defined over~$\FF_q$. Let~$\ell$ be a prime number dividing the number of
$\FF_q$-rational points on the Jacobian~$\jac$, and with $\ell\nmid q$. Depending on the cases in
table~\ref{table:conditions} we get the following properties of~$\jac$.
    \begin{description}
    \item[Case~\textsc{i}] $-q^2\equiv q^4\equiv 1\pmod{\ell}$ and~$\jac[\ell]\subseteq\jac(\FF_{q^4})$. If $\ell\neq 2$, then~$\jac(\FF_q)[\ell]$ is cyclic.
    \item[Case~\textsc{ii}] $q^3\equiv 1\pmod{\ell}$, $\jac[\ell]\subseteq\jac(\FF_{q^6})$ and~$\jac(\FF_q)$ is cyclic.
    If $\ell\neq 3$, then $q\not\equiv 1\pmod{\ell}$.
    \item[Case~\textsc{iii}] $-q^3\equiv q^6\equiv 1\pmod{\ell}$ and~$\jac[\ell]\subseteq\jac(\FF_{q^6})$. If $\ell\neq 3$, then~$\jac(\FF_q)[\ell]$ is cyclic.
    \item[Case~\textsc{iv}] $q\not\equiv q^5\equiv 1\pmod{\ell}$, $\jac[\ell]\subseteq\jac(\FF_{q^{10}})$ and~$\jac(\FF_q)$ is cyclic.
    \item[Case~\textsc{v}] $q\not\equiv q^5\equiv 1\pmod{\ell}$, $\jac[\ell]\subseteq\jac(\FF_{q^{10}})$ and~$\jac(\FF_q)$ is cyclic.
    \item[Case~\textsc{vi}] $-q^6\equiv q^{12}\equiv 1\pmod{\ell}$, $\jac[\ell]\subseteq\jac(\FF_{q^{24}})$ and~$\jac(\FF_q)$ is cyclic.
    \item[Case~\textsc{vii}] $q\equiv 1\pmod{\ell}$ and~$\jac[\ell]\subseteq\jac(\FF_{q^2})$. If $\ell\neq 2$, then~$\jac(\FF_q)[\ell]$ is bicyclic.
    \item[Case~\textsc{viii}] $-q\equiv q^2\equiv 1\pmod{\ell}$ and~$\jac[\ell]\subseteq\jac(\FF_{q^2})$. If $\ell\neq 2$, then~$\jac(\FF_q)[\ell]$ is bicyclic.
    \item[Case~\textsc{ix}] If $\ell\neq 3$, then $q\not\equiv q^3\equiv 1\pmod{\ell}$, $\jac[\ell]\subseteq\jac(\FF_{q^3})$ and~$\jac(\FF_q)[\ell]$ is bicyclic.
    \end{description}
\end{theorem}

\begin{corollary}
If $\ell>3$, then the full embedding degree with respect to $\ell$ of the Jacobian~$\jac$ of a supersingular genus two
curve defined over~$\FF_q$ is at most $24$, and~$\jac(\FF_q)[\ell]$ is of rank at most two as a $\ZZ/\ell\ZZ$-module.
\end{corollary}

\begin{proof}[Proof of Theorem~\ref{theorem:TEDss}]
In the following we consider each case in table~\ref{table:conditions} separately. Throughout this proof, assume that
    $$f(X)=X^4+sX^3+tX^2+sqX+q^2$$
is the Weil polynomial of the Jacobian~$\jac$ of some supersingular genus two curve~$C$ defined over the finite
field~$\FF_q$ of characteristic $p$, and let $\ell$ be a prime number dividing $f(1)$.

\subsection*{The case $s=0$}\label{subsec:s=0}

First consider the cases~\textsc{i}, \textsc{ii}, \textsc{iii}, \textsc{vii} and \textsc{viii} of
table~\ref{table:conditions}.

\subsubsection*{Case \textsc{i}}

If $(s,t)=(0,0)$, then $f(1)=1+q^2\equiv 0\pmod{\ell}$, i.e. $q^2\equiv -1\pmod{\ell}$. So $f(X)\equiv
X^4-1\pmod{\ell}$, $q^4\equiv 1\pmod{\ell}$ and~$\jac[\ell]\subseteq\jac(\FF_{q^4})$. $\tau=2q$ in
Theorem~\ref{theorem:main}, so if $\ell\neq 2$, then~$\jac(\FF_q)[\ell]$ is cyclic.

\subsubsection*{Case \textsc{ii}}

If $(s,t)=(0,q)$, then the roots of $f$ modulo $\ell$ are given by $\pm 1$ and $\pm q$. Since $f(1)=q^2+q+1\equiv
0\pmod{\ell}$, we know that $q\equiv\frac{1}{2}(-1\pm\sqrt{-3})\pmod{\ell}$. It follows that $q^3\equiv 1\pmod{\ell}$
and~$\jac[\ell]\subseteq\jac(\FF_{q^6})$. If $\ell=2$, then $p\neq 2$, and $f(1)$ is odd. So $\ell\neq 2$. $\tau=q$ in
Theorem~\ref{theorem:main}, so~$\jac(\FF_q)$ is cyclic.

\subsubsection*{Case \textsc{iii}}

If $(s,t)=(0,-q)$, then the roots of $f$ modulo $\ell$ are given by $\pm 1$ and $\pm q$. Since $f(1)=q^2-q+1\equiv
0\pmod{\ell}$, we know that $q\equiv\frac{1}{2}(1\pm\sqrt{-3})\pmod{\ell}$. It follows that $q^6\equiv 1\pmod{\ell}$
and~$\jac[\ell]\subseteq\jac(\FF_{q^6})$. As in case \textsc{ii}, $\ell\neq 2$. Now $\tau=3q$, so if $\ell\neq 3$, then
$\jac(\FF_q)[\ell]$ is cyclic.

\subsubsection*{Case \textsc{vii}}

If $(s,t)=(0,-2q)$, then $q\equiv 1\pmod{\ell}$ and $f(X)=(X^2-q)^2$. Since~$q$ is an odd power of~$p$, $X^2-q$ is
irreducible over $\QQ$. So by \cite[Theorem~2]{tate}, $\jac\simeq E\times E$ for some supersingular elliptic curve~$E$.
It follows that $\jac[\ell]\subseteq\jac(\FF_{q^2})$. $\tau=4q$, so if $\ell\neq 2$, then~$\jac(\FF_q)[\ell]$ is
bicyclic.

\subsubsection*{Case \textsc{viii}}

If $(s,t)=(0,2q)$, then $q\equiv -1\pmod{\ell}$ and $f(X)=(X^2+q)^2$. Since $X^2+q$ is irreducible over $\QQ$, it
follows that~$\jac\simeq E\times E$ for some supersingular elliptic curve~$E$. So $q^2\equiv 1\pmod{\ell}$
and~$\jac[\ell]\subseteq\jac(\FF_{q^2})$. $\tau=0$ and $\omega=i\sqrt{q}$ is a $q$-Weil number of $\jac$. Since $q$ is
an even power of~$p$, $K=\QQ(\omega)=\QQ(i)$ is of discriminant $d_K=-4$. Hence, if $\ell\neq 2$,
then~$\jac(\FF_q)[\ell]$ is bicyclic by Theorem~\ref{theorem:main:supplement}.

\subsection*{Case \textsc{iv}--\textsc{vi}}\label{subsec:cyclic}

Now we consider the cases~\textsc{iv}, \textsc{v} and \textsc{vi} of table~\ref{table:conditions}.

\subsubsection*{Case \textsc{iv}}

If $(s,t)=(\sqrt{q},q)$, then $4\tau=5q$ in Theorem~\ref{theorem:main}. Since $f(1)$ is odd, we know that $\ell\neq 2$.
If $\ell$ divides $4\tau$, then $\ell=5$; $\ell\nmid q$, since~$C$ is supersingular. But then
$f(1)=q^2+q\sqrt{q}+q+\sqrt{q}+1\equiv 0\pmod{5}$, i.e. $q\equiv 2\pmod{5}$. Since $a$ is even and $2$ is not a
quadratic residue modulo $5$, this is impossible. So $\ell\nmid 4\tau$. If $q\equiv 1\pmod{\ell}$, then $f(1)\equiv
5\pmod{\ell}$, i.e. $\ell=5$. But then $\ell$ divides $4\tau$, a contradiction. So~$\jac(\FF_q)$ is cyclic by
Theorem~\ref{theorem:main}. From $f(1)\equiv 0\pmod{\ell}$ it follows that $q^5\equiv 1\pmod{\ell}$. Since the complex
roots of $f$ are of the form $\sqrt{q}\xi$, where $\xi$ is a primitive $5^{\text{th}}$ root of unity, it follows that
$\jac[\ell]\subseteq\jac(\FF_{q^{10}})$. The case $(s,t)=(-\sqrt{q},q)$ follows similarly.

\subsubsection*{Case \textsc{v}}

If $(s,t)=(\sqrt{5q},3q)$ and $p=5$, then $4\tau$ is a power of $5$ in Theorem~\ref{theorem:main}. Since $f(1)$ is odd,
we know that $\ell\neq 2$. If $\ell$ divides $4\tau$, then $\ell=5$. Since~$C$ is supersingular and defined over a
field of characteristic $p=5$, this is a contradiction. So $\ell\nmid 4\tau$. If $q\equiv 1\pmod{\ell}$, then
$f(1)\equiv 5+2\sqrt{5}\equiv 0\pmod{\ell}$, and it follows that $\ell=5$. So~$\jac(\FF_q)$ is cyclic by
Theorem~\ref{theorem:main}. From $f(1)\equiv 0\pmod{\ell}$ it follows that $q^5\equiv 1\pmod{\ell}$. Since the complex
roots of $f$ are of the form $\sqrt{q}\xi$, where $\xi$ is a primitive $10^{\text{th}}$ root of unity, it follows that
$\jac[\ell]\subseteq\jac(\FF_{q^{10}})$. The case $(s,t)=(-\sqrt{5q},3q)$ follows similarly.

\subsubsection*{Case \textsc{vi}}

If $(s,t)=(\sqrt{2q},q)$ and $p=2$, then $4\tau=3\cdot 2^a$ for some number~$a\in\NN$. Hence, if $\ell$ divides
$4\tau$, then $\ell=3$. But $3\nmid f(1)$; thus, $\ell\nmid 4\tau$. If $q\equiv 1\pmod{\ell}$, then $f(1)\equiv
3+2\sqrt{2}\equiv 0\pmod{\ell}$, and it follows that $\ell=1$. So~$\jac(\FF_q)$ is cyclic by
Theorem~\ref{theorem:main}. From $f(1)\equiv 0\pmod{\ell}$ it follows that $q^6\equiv -1\pmod{\ell}$. Since the complex
roots of $f$ are of the form $\sqrt{q}\xi$, where $\xi$ is a primitive $24^{\text{th}}$ root of unity, it follows
that~$\jac[\ell]\subseteq\jac(\FF_{q^{24}})$. The case $(s,t)=(-\sqrt{2q},q)$ follows similarly.

\subsection*{Case \textsc{ix}}\label{subsec:bicyclic}

Finally, consider the case~\textsc{ix}. Assume that $(s,t)=(-2\sqrt{q},3q)$. We see that $f(X)=g(X)^2$, where
$g(X)=X^2-\sqrt{q}X+q$. Since the complex roots of $g$ are given by $\frac{1}{2}(1\pm\sqrt{-3})\sqrt{q}$, $g$ is
irreducible over $\QQ$. So by~\cite[Theorem~2]{tate}, $\jac\simeq E\times E$ for some supersingular elliptic curve~$E$.
Hence, either~$\jac(\FF_q)[\ell]$ is bicyclic or equals the full $\ell$-torsion subgroup of~$\jac$.

Assume~$\jac(\FF_q)[\ell]=\jac[\ell]$. Then $q\equiv 1\pmod{\ell}$, i.e. $\sqrt{q}\equiv\pm 1\pmod{\ell}$. But then
$f(1)\equiv 9\equiv 0\pmod{\ell}$ or $f(1)\equiv 1\equiv 0\pmod{\ell}$, i.e. $\ell=3$.

Since $f(1)=(1-\sqrt{q}+q)^2\equiv 0\pmod{\ell}$, we know that $q\equiv \frac{1}{2}(-1\pm\sqrt{-3})\pmod{\ell}$. So
$q^3\equiv 1\pmod{\ell}$. Since $\ell\neq 3$, it follows that $q\not\equiv 1\pmod{\ell}$. Hence,
$\jac[\ell]\subseteq\jac(\FF_{q^3})$ by the non-degeneracy of the Tate-pairing.

The case $(s,t)=(2\sqrt{q},3q)$ follows similarly.
\end{proof}

\bibliographystyle{plain}
\bibliography{references}

\begin{thebibliography}{10}

\bibitem{balasubramanian}
R.~Balasubramanian and N.~Koblitz.
\newblock The improbability that an elliptic curve has subexponential discrete
  log problem under the menezes-okamoto-vanstone algorithm.
\newblock {\em J. Cryptology}, 11:141--145, 1998.

\bibitem{boneh-franklin}
D.~Boneh and M.~Franklin.
\newblock Identity-based encryption from the weil pairing.
\newblock {\em SIAM J. Computing}, 32(3):586--615, 2003.

\bibitem{cassels}
J.W.S. Cassels and E.V. Flynn.
\newblock {\em Prolegomena to a Middlebrow Arithmetic of Curves of Genus~$2$}.
\newblock London Mathematical Society Lecture Note Series. Cambridge University
  Press, 1996.

\bibitem{hhec}
G.~Frey and T.~Lange.
\newblock Varieties over special fields.
\newblock In H.~Cohen and G.~Frey, editors, {\em Handbook of Elliptic and
  Hyperelliptic Curve Cryptography}, pages 87--113. Chapman \& Hall/CRC, 2006.

\bibitem{galbraith01}
S.D. Galbraith.
\newblock Supersingular curves in cryptography.
\newblock In {\em Advances in Cryptology -- Asiacrypt 2001}, volume 2248 of
  {\em Lecture Notes in Computer Science}, pages 495--513. Springer, 2001.

\bibitem{galbraith05}
S.D. Galbraith.
\newblock Pairings.
\newblock In I.F. Blake, G.~Seroussi, and N.P. Smart, editors, {\em Advances in
  Elliptic Curve Cryptography}, volume 317 of {\em London Mathematical Society
  Lecture Note Series}, pages 183--213. Cambridge University Press, 2005.

\bibitem{galbraith07}
S.D. Galbraith, F.~Hess, and F.~Vercauteren.
\newblock Hyperelliptic pairings.
\newblock In {\em Pairing 2007}, Lecture Notes in Computer Science, pages
  108--131. Springer, 2007.

\bibitem{hess}
F.~Hess.
\newblock A note on the tate pairing of curves over finite fields.
\newblock {\em Arch. Math.}, 82:28--32, 2004.

\bibitem{howeNartRitzenthaler}
E.W. Howe, E.~Nart, and C.~Ritzenthaler.
\newblock Jacobians in isogeny classes of abelian surfaces over finite fields,
  2007.
\newblock Preprint, available at \url{http://arxiv.org}.

\bibitem{koblitz87}
N.~Koblitz.
\newblock Elliptic curve cryptosystems.
\newblock {\em Math. Comp.}, 48:203--209, 1987.

\bibitem{koblitz89}
N.~Koblitz.
\newblock Hyperelliptic cryptosystems.
\newblock {\em J. Cryptology}, 1:139--150, 1989.

\bibitem{lang59}
S.~Lang.
\newblock {\em Abelian Varieties}.
\newblock Interscience, 1959.

\bibitem{maisner-nart}
D.~Maisner and E.~Nart with an appendix~by Everett W.~Howe.
\newblock Abelian surfaces over finite fields as jacobians.
\newblock {\em Experimental Mathematics}, 11(3):321--337, 2002.

\bibitem{miller-algorithm}
V.S. Miller.
\newblock Short programs for functions on curves, 1986.
\newblock Unpublished manuscript, available at
  \url{http://crypto.stanford.edu/miller/miller.pdf}.

\bibitem{miller}
V.S. Miller.
\newblock The weil pairing, and its efficient calculation.
\newblock {\em J. Cryptology}, 17:235--261, 2004.

\bibitem{neukirch}
J.~Neukirch.
\newblock {\em Algebraic Number Theory}.
\newblock Springer, 1999.

\bibitem{ravnshoj}
C.R. Ravnshøj.
\newblock Non-cyclic subgroups of {J}acobians of genus two curves with complex
  multiplication, 2007.
\newblock Preprint presented at AGCT~11, available at \url{http://arxiv.org}.
  Submitted to \emph{Proceedings of AGCT~11}.

\bibitem{rubin-silverberg02}
K.~Rubin and A.~Silverberg.
\newblock Supersingular abelian varieties in cryptology.
\newblock In M.~Yung, editor, {\em CRYPTO 2002}, Lecture Notes in Computer
  Science, pages 336--353. Springer, 2002.

\bibitem{rubin-silverberg07}
K.~Rubin and A.~Silverberg.
\newblock Using abelian varieties to improve pairing-based cryptography, 2007.
\newblock Preprint, available at
  \url{http://www.math.uci.edu/~asilverb/bibliography/}.

\bibitem{sil}
J.H. Silverman.
\newblock {\em The Arithmetic of Elliptic Curves}.
\newblock Springer, 1986.

\bibitem{stichtenoth}
H.~Stichtenoth and C.~Xing.
\newblock On the structure of the divisor class group of a class of curves over
  finite fields.
\newblock {\em Arch. Math.}, 65:141--150, 1995.

\bibitem{tate}
J.~Tate.
\newblock Endomorphisms of abelian varieties over finite fields.
\newblock {\em Invent. Math.}, 2:134--144, 1966.

\bibitem{zhu}
H.J. Zhu.
\newblock Group structures of elementary supersingular abelian varieties over
  finite fields.
\newblock {\em J. Number Theory}, 81:292--309, 2000.

\end{thebibliography}

\end{document}